\documentclass[12pt,a4paper]{amsart}
\usepackage[utf8]{inputenc}
\usepackage{amsmath}
\usepackage{amsfonts}
\usepackage[colorlinks,linktocpage=true]{hyperref}
\hypersetup{colorlinks,linkcolor={blue},citecolor={blue},urlcolor={black}}
\usepackage{cite}
\usepackage{array,booktabs}
\usepackage{amssymb}
\usepackage{esint}
\theoremstyle{plain}
\newtheorem{theorem}{Theorem}[section]

\newtheorem{corollary}[theorem]{Corollary}
\newtheorem{lemma}[theorem]{Lemma}

\theoremstyle{definition}
\newtheorem{definition}{Definition}
\theoremstyle{remark}
\newtheorem{remark}{Remark}

\newtheorem*{acknowledgements}{Acknowledgements}

\DeclareMathOperator*{\osc}{osc}
\DeclareMathOperator*{\dist}{dist}

\def\R{\mathbb R}

\def\N{\mathbb N}

\usepackage[showonlyrefs]{mathtools}
\mathtoolsset{showonlyrefs}

\numberwithin{equation}{section}
\author{Lauri Hitruhin}
\address{University of Helsinki, Department of Mathematics and Statistics, P.O. Box 68, FIN-00014 University of Helsinki, Finland}
\email{lauri.hitruhin@helsinki.fi}
\author{Athanasios Tsantaris}
\address{University of Helsinki, Department of Mathematics and Statistics, P.O. Box 68, FIN-00014 University of Helsinki, Finland}
\email{athanasios.tsantaris@helsinki.fi}
\makeatletter
\@namedef{subjclassname@2020}{%
	\textup{2020} Mathematics Subject Classification}

\makeatother

\title[Finite distortion curves]{Finite distortion curves: Continuity, Differentiability and Lusin's (N) property}

\begin{document}
\let\thefootnote\relax\footnotetext{This work was supported by the Academy of Finland projects \#332671 and SA-1346562 }
	\maketitle
\begin{abstract}
	We define finite distortion $\omega$-curves and we show that for some forms $\omega$ and when the distortion function is sufficiently exponentially integrable the map is continuous, differentiable almost everywhere and has Lusin's (N) property.  This is achieved through some higher integrability results about finite distortion $\omega$-curves. It is also shown that this result is sharp both for continuity and for Lusin's (N) property. We also show that if we assume weak monotonicity for the coordinates of a finite distortion $\omega$-curve we obtain continuity. 
\end{abstract}
	\section{Introduction}
	The study of mappings of \emph{finite distortion} is a central theme in  geometric function theory. Their importance stems in part from the various connections with other fields like holomorphic dynamics, PDEs and non-linear elasticity to name a few.	A mapping $f\in W_{loc}^{1,1}(\Omega,\R^n)$, where $\Omega$ is a domain in $\R^n$, is said to be of finite distortion if $J_f\in L^1_{loc}$ and there exists a measurable function $K_f(x)\geq 1$ which is finite almost everywhere  such that the  following inequality is satisfied 
	\begin{equation}\label{eq01}
		|Df(x)|^n\leq K_f(x) J_f(x) \ \text{a.e.},
	\end{equation}
where $|Df|$ denotes the sup norm of the differential $Df$ and $J_f$ the Jacobian determinant. When $K_f(x)\leq K<\infty$ almost everywhere, where $K$ constant, such a map is called \emph{quasiregular} (also known as maps of bounded distortion). 

Quasiregular maps are the correct higher dimensional generalization of holomophic maps in the complex plane. Starting with Reshetnyak, they have been studied extensively since the sixties and by now their theory is well developed. We refer to the books  \cite{Reshetnyak1989,Rickman} for more details on quasiregular maps. In many applications the distortion must be allowed to blow up in a controlled manner. This has led to  systematic study of the more general class of mappings with finite distortion. By now there is a quite extensive literature for these maps and we refer to the books \cite{Hencl2014, Iwaniec2001} for their general theory. 

One of the basic facts about maps of finite distortion concerns their continuity.  In \cite{Vodopyanov1977} Vodopyanov and Goldshtein showed that a mapping of finite distortion $f\colon\Omega\to \R^n$ that belongs in the Sobolev space $W_{loc}^{1,n}(\Omega,\R^n)$ has a continuous representative in the same Lebesgue class. Notice that a finite distortion mapping does not necessarily belong in  $W_{loc}^{1,n}(\Omega,\R^n)$. However, in \cite[Theorem 7.1]{Iwaniec2002} it was shown that if we assume $e^{\lambda K_f(x)}\in L^1_{loc}$, for some large enough $\lambda>0$ then we have  that $f$ is in $W_{loc}^{1,n}(\Omega,\R^n)$ and as a result it has a  continuous representative. In fact, in \cite{Iwaniec2001a} it was shown that  $e^{\lambda K_f(x)}\in L^1_{loc}$, for any $\lambda>0$ is enough in order to obtain continuity. Moreover, mappings of finite distortion with $e^{\lambda K_f(x)}\in L^1_{loc}$, for some $\lambda>0$, are almost everywhere differentiable and have Lusin's (N) property, i.e. they map sets of zero (Lebesgue) measure to sets of zero measure, see \cite{Hencl2014} and also \cite{Kauhanen2001,Kauhanen2003}. 

Recently there has been a lot of interest in generalizing quasiregular mappings in the setting where the range and the domain of definition do not have the same dimension, see \cite{Pankka2020,Onninen2021,Heikkilae2021,Heikkilae2021a}. Pankka in \cite{Pankka2020} called such maps \emph{quasiregular curves} in accordance with their holomorphic counterparts which are called holomorphic curves. In this paper we study  under what conditions do \textit{curves of finite distortion}, which are a generalization of quasiregular curves, have the desirable properties we mentioned above. By $\Omega^n(\R^m)$ we denote the space of smooth differential $n$-forms in $\R^m$, where $n\leq m$. In the following definition $\omega\in \Omega^n(\R^m)$ is a smooth, non-vanishing and closed differential $n$-form. We call such forms $n$-\textit{volume forms}.
\begin{definition}[Finite distortion $\omega$-curve]\label{def1}
	A mapping $f\in W^{1,1}_{loc}(\Omega,\R^m)$, where $\Omega$ is a domain in $\R^n$ and $n\leq m$ is called a finite distortion $\omega$-curve if $\star f^*\omega\in L^1_{loc}$ and there exists  a measurable function $K_f(x)\geq 1$ which is finite almost everywhere  such that the  following inequality is satisfied 
	\begin{equation}\label{eq03}
		(||\omega||\circ f) |Df(x)|^n\leq K_f(x) (\star f^*\omega),
	\end{equation}
almost everywhere on $\Omega$.
\end{definition} 
Here $\star f^*\omega$ is the Hodge star of the $n$-volume form $f^*\omega$ (the pullback of $\omega$), i.e. the function that satisfies $(\star f^*\omega) dx_1\wedge\dots\wedge dx_n=f^*\omega$. The function $||\omega||\colon \R^m\to[0,\infty)$ is the pointwise comass norm of $\omega$ defined as  \begin{equation}\label{eq04}
	||\omega||(p)=\sup\{|\omega_p(v_1,\dots,v_n)|\colon \ v_1,\dots,v_2\in \R^m, \ |v_i|\leq 1\}
\end{equation}
for each $p\in\R^m$.

 When $K_f(x)\leq K<\infty$, where $K$ is some constant, we get the class of quasiregular $\omega$-curves. We note here that in \cite{Pankka2020} quasiregular curves are assumed to be continuous by definition. However in \cite{Onninen2021} Pankka and Onninen showed that the continuity assumption can be dropped from the definition when the form $\omega$ has constant coefficients.

Our first result shows that, unlike the case of mappings of finite distortion, $\exp(\lambda K_f(x)) \in L^1_{loc}$, for any $\lambda>0$ is not enough to guarantee the existence of a continuous representative even for the most simple forms. In fact it shows that even  $f\in W^{1,n}_{loc}(\Omega,\R^m)$ is not enough (compare this with \cite[Theorem 1.4]{Iwaniec2001a}).  
\begin{theorem}\label{thm2}
	For every $\lambda<2$ there exists a finite distortion $\omega$-curve, $f:\R^2\to\R^3$ for $\omega=dx\wedge dy$ such that $f\in W^{1,2}_{loc}(\R^2,\R^3)$ and $\exp(\lambda K_f(x))\in L^1_{loc}$, but $f$ does not have a continuous representative.
\end{theorem}
 The situation is similar for Lusin's (N) condition as well. The condition $\exp(\lambda K_f(x)) \in L^1_{loc}$, for any $\lambda>0$ is not enough.
\begin{theorem}\label{thmlusin}
	For every $\lambda<2$ there exists a finite distortion $\omega$-curve, $f:\R^2\to\R^4$ for $\omega=dx\wedge dy$ such that  $f\in W^{1,2}_{loc}(\R^2,\R^4)$ and $\exp(\lambda K_f(x))\in L^1_{loc}$, but $f$ does not have Lusin's (N) property.
\end{theorem}
To explain why finite distortion curves differ from finite distortion maps in terms of their continuity properties we need to introduce the important concept of weak monotonicity. A real valued function $u\in W_{loc}^{1,1}(\Omega)$ is called weakly monotone if the following holds: For all balls $B$ compactly contained in $\Omega$ and all constants $m,M\in\R$, $m\leq M$ such that

\begin{equation}
	|M-u|-|u-m|+2u-m-M\in W^{1,1}_0(B)
\end{equation} we have that \[m\leq u(x)\leq M,\] for almost every $x\in B$. Here $W^{1,1}_0(B)$ stands for the closure of $C^\infty_0(B)$, the compactly supported smooth functions, with respect to the Sobolev space norm. For continuous functions this is equivalent to the more well known notion of monotonicity, namely 
\begin{equation}\osc_B u\leq \osc_{\partial B}u, 
\end{equation}
where $\osc_B u=\sup_{x,y\in B}|u(x)-u(y)|$.

The reason that weak monotonicity is important when discussing continuity is that the coordinate functions of a finite distortion map in a suitable Orlicz-Sobolev space, are weakly monotone, see for example \cite[Theorem 7.3.1]{Iwaniec2001}. It is also well known that weakly monotone functions in that same Orlicz-Sobolev space have continuous representatives (for  more precise formulations see  Lemma \ref{lemma0} in section \ref{prelim}).  Thus, roughly speaking, to show the existence of a continuous representative for a mapping of finite distortion it is enough to show that the map has sufficient regularity and its coordinates are weakly monotone.

However, the situation is different for curves of finite distortion. Their coordinate functions are not necessarily weakly monotone as our example in Theorem \ref{thm2} shows. On the other hand, if we assume the weak monotonicity of the coordinate maps and that $e^{\lambda K_f(x)}\in L^1_{loc}$, for some $\lambda>0$, then by adapting the arguments in \cite[Theorem 2.4]{Hencl2014} and \cite[Theorem 7.5.2]{Iwaniec2001} we obtain the existence of a continuous representative. 

\begin{theorem}\label{thm0}
	Let $f=(f_1,\dots,f_m)\colon \Omega\to \R^m$ be a finite distortion $\omega$-curve, for a bounded $n$-volume form $\omega$ in $\R^m$, such that the coordinate functions $f_i$, $i=1,\dots,m$ are weakly monotone. Suppose that there exists a $\lambda>0$ such that $e^{\lambda K_f(x)}\in L^1_{loc}$. Then $f$ has a continuous representative in $\Omega$.  	
\end{theorem}

Our next result gives some sufficient conditions in order for a finite distortion $\omega$-curve to have a continuous representative which is also almost everywhere differentiable and has Lusin's (N) property in the simple case of constant coefficient forms. We will call an $n$-volume form in $\R^m$,  $\omega=\sum_I \phi_I(x)dx_I$, where the sum is over all multi-indices $I=(i_1,\dots, i_n)$, $1\leq i_1< \dots< i_n\leq m$,   a \textit{constant coefficient form} when the functions $\phi_I$ are constant. Here $dx_I$ denotes the $n$-covector $dx_{i_1}\wedge\dots\wedge dx_{i_n}$.
\begin{theorem}\label{thm3/4}
	Let $f:\Omega\to \R^m$ be a finite distortion $\omega$-curve, for a constant coefficient $n$-volume form $\omega$ in $\R^m$. There exists a $c_2=c_2(n)>0$ such that if $\exp(\lambda K_f(x))\in L^1_{loc}$, for some $\lambda>c_2$, then $f$ has a continuous and almost everywhere differentiable representative. Moreover, this representative has Lusin's (N) property, that is, it maps sets of zero $n$-dimensional Lebesgue measure to sets of zero $n$-dimensional Lebesgue measure. 
\end{theorem} 
Our results in Theorems \ref{thm2} and \ref{thmlusin} essentially show that the integrability condition on the distortion cannot be relaxed.

Next we consider more general classes of forms.  We will call an $n$-volume form in $\R^m$,  $\omega=\sum_I \phi_I(x)dx_I$ \textit{bounded} if all the non-zero, real valued functions $\phi_I(x)$ are bounded in $\R^m$ and there is a constant $c>0$  such that, for all $x\in \R^m$, there is a multi index $I_x$ such that $|\phi_{I_x}(x)|>c$.  Notice that for a bounded form $\omega$, there is a constant $c>0$ so that $||\omega||(x)>c$, for all $x\in\R^m$.  We will use the symbols $|\omega|_{\ell_1}$, $|\omega|_{inf}$ to denote  \[\sup_{x\in\R^m}\sum_I |\phi_I(x)|\ \text{and}\ \inf\{||\omega||(x)\colon x\in\R^m\}\] respectively. When the functions $\phi_I$ are constant we say that $\omega$ is  a constant coefficient form. Notice that for bounded forms we have that $|\omega|_{\ell_1}<\infty$ and $|\omega|_{inf}>0$. For an $n$-volume form  $\omega=\sum_I \phi_I(x)dx_I$ we will also denote by $H_\omega$ the set of multi-indices $I$ such that $\phi_I(x)$ is not identically zero. 

We define now $E_f$ to be the following set
\begin{equation}\label{eq05}
	\begin{aligned}
		E_f=\{x\in\Omega\colon &\text{there exists a sequence}\ x_n\to x \ \text{such that}\ \\& \star f^*dx_I(x_n)\to\infty, \ \text{for some multi-index} \ I\in H_\omega\}.
	\end{aligned}
\end{equation} 
Notice that $\star f^*dx_I$ is a real valued function for all $I$.
For a bounded $n$-form $\omega$ we say that the pair $(f,\omega)$ satisfies \textit{condition (D)} if the set $E_f$ can be covered by a countable number of open balls $\{B_i\}_{i=1}^\infty$  and on each ball $B_i$ there exists a multi index $I_i$ such that $\star f^*dx_{I_i}\geq \max_I\{\star f^*dx_I\}$.

\begin{theorem}\label{thm1}
	Let $f:\Omega\to \R^m$ be a finite distortion $\omega$-curve, for a bounded $n$-volume form $\omega$ in $\R^m$, such that the pair $(f,\omega)$ satisfies condition (D). Suppose that $c=c(n,\omega)=\frac{n|\omega|_{\ell_1}}{c_1|\omega|_{inf}}$, where $c_1=c_1(n)>0$ constant. If $\exp(\lambda K_f(x))\in L^1_{loc}$, for some $\lambda>c$, then there is a continuous and almost everywhere differentiable representative of $f$. Moreover, this representative has Lusin's (N) property, that is, it maps sets of zero $n$-dimensional Lebesgue measure to sets of zero $n$-dimensional Lebesgue measure. 
\end{theorem}

\begin{remark}
	The constant $c_1$ in  Theorem \ref{thm1} comes from a well known result about the higher integrability if the Jacobian of mappings of finite distortion proven in \cite{Faraco2005}, see Lemma \ref{lemma2} in section \ref{prelim} for the precise statement. It is known that in two dimensions $c_1=1$ (see \cite[Theorem 20.4.12]{Astala2009}) but the exact value of this constant remains unknown in higher dimensions. Notice that when $n=2$ and the form $\omega=dx\wedge dy$  we have that the constant above is $c=2$. Thus Theorems \ref{thm2} and \ref{thmlusin} imply that Theorem \ref{thm1} is sharp in this  case.
\end{remark}
An immediate  corollary of the above theorems, which is worth pointing out, is the area formula (see \cite{StevenG.Krantz2008}). We note here that $\mathcal{H}^n$ denotes the $n$-dimensional Hausdorff measure in $\R^m$. 
\begin{corollary}[Area formula]
	Let $f$ be as in Theorem \ref{thm1} or \ref{thm3/4} then 
	\begin{equation}
		\int_E h(x)\sqrt{\det (Df^TDf)}dx=\int_{f(E)}\sum_{x\in f^{-1}(y)}h(x)d\mathcal{H}^n(y),
	\end{equation}
holds for all measurable functions $h:\R^n\to [0,\infty]$ and all measurable sets $E\subset\R^n$.

\end{corollary}

Theorems \ref{thm1} and \ref{thm3/4} follow from the fact that curves of finite distortion satisfying the assumptions of  the theorems enjoy a higher than natural amount of regularity.  We state the precise results since they are of independent interest.
\begin{theorem}\label{thm1/2}
	Let $f:\Omega\to \R^m$ be a finite distortion $\omega$-curve, for a bounded $n$-volume form $\omega$ in $\R^m$, such that the pair $(f,\omega)$ satisfies condition (D).  If $\exp(\lambda K_f(x))\in L^1_{loc}(\Omega)$, for some $\lambda>0$, then there exists a constant $c_1=c_1(n)>0$ such that
	\begin{equation}
		\star f^*\omega\log^a(e+\star f^*\omega)\in L^1_{loc}(\Omega) \ \text{and} \ |Df(x)|^n\log^{a-1}(e+|Df(x)|)\in L^1_{loc}(\Omega),
	\end{equation}
 for  all $a<c_1\lambda \frac{|\omega|_{inf}}{|\omega|_{\ell_1}}$.
\end{theorem} 

	\begin{theorem}\label{thmhigh}
	Let $f:\Omega\to \R^m$ be a finite distortion $\omega$-curve, where $\omega$ is an $n$-volume form with constant coefficients.  If $\exp(\lambda K_f(x))\in L^1_{loc}(\Omega)$, for some $\lambda>0$, and $f\in W^{1,n}_{loc}(\Omega,\R^m)$ then there is a $c_3=c_3(n)>0$ such that
	\begin{equation}
		\star f^*\omega\log^a(e+\star f^*\omega)\in L^1_{loc}(\Omega) \ \text{and} \ |Df(x)|^n\log^{a-1}(e+|Df(x)|)\in L^1_{loc}(\Omega),
	\end{equation}
	for all $a< c_3\lambda$.
\end{theorem}
 It is also interesting to point out that the higher integrability implies also modulus of continuity estimates. This can be achieved through an embedding result for Orlicz-Sobolev spaces of Donaldson and Trudinger \cite[Theorem 3.6]{Donaldson1971}, see also \cite[Theorem 8.40]{RobertA.Adams2003}.
 \begin{theorem}\label{thmmodulus}
 	Let $f:\Omega\to \R^m$ be a finite distortion $\omega$-curve, where $\omega$ is either an $n$-volume form with constant coefficients or a bounded $n$-volume form such that the pair $(f,\omega)$ satisfies condition (D). Then there exists a constant $q=q(n,\omega)>0$ such that if $\exp(\lambda K_f(x))\in L^1_{loc}(\Omega)$, for some $\lambda>q$ then for every compact subdomain $F\subset \Omega$ and every $x,y\in F$ we have
 	\begin{equation}\label{eqmodulus}
 		|f(x)-f(y)|\leq Q ||f||_{W^{1,P}}\int_{|x-y|^{-n}}^\infty \frac{P^{-1}(t)}{t^{(n+1)/n}}dt,
 	\end{equation} 
 where $P(t)=t^n\log ^a(e+t)$, $a>n$, $Q=Q(n,F)>0$ constant and $||f||_{W^{1,P}}=||f||_P+||Df||_P$.
 \end{theorem}

The rest of the paper is organized as follows. In section \ref{prelim} we introduce the necessary notation and terminology and we recall results from the theory of finite distortion maps that we are going to need. In section \ref{higherint} we prove the higher integrability results of Theorems \ref{thm1/2} and  \ref{thmhigh} while in section \ref{proofs} we give the proofs of Theorems \ref{thm1} and \ref{thm3/4}. In section \ref{monotonicity} we prove Theorem \ref{thm0}. Finally, in section \ref{counterexamples} we construct the functions of Theorems \ref{thm2} and \ref{thmlusin}.

\begin{acknowledgements}
	We would like to thank Pekka Pankka for a question which motivated this work and for many interesting discussions and comments. 
	\end{acknowledgements}

	\section{Preliminaries on Orlicz-Sobolev spaces and finite distortion maps}\label{prelim}

	Here we collect results and terminology from Orlicz-Sobolev spaces and finite distortion maps that we are going to need for the proofs of our results, we refer to \cite{Iwaniec2001, RobertA.Adams2003} for more details. 
	
	First we need to introduce the \textit{Orlicz spaces} and the \textit{Orlicz-Sobolev spaces}. An Orlicz function is a continuous and increasing function $P:[0,\infty)\to[0,\infty)$ with $P(0)=0$ and $\lim_{t\to\infty}P(t)=\infty$. We will also assume that the function $P$ is convex. The Orlicz space $L^P(\Omega)$ consists of all measurable functions $u\colon \R^n\to\mathbb{R}$ such that 
	\begin{equation}\label{eq11}
	||u||_P:=\int_\Omega P(\lambda |u|)<\infty, \ \text{for some}\ \lambda=\lambda(u)>0.
	\end{equation}
The functional $||\cdot||_P$, which is usually called the Luxemburg functional, is a norm in $L^P$ and in fact $L^P$ is a Banach space with this norm.

The Orlicz space $L^P$ for $P(t)= t^p\log^a(e+t)$, $1\leq p< \infty$, $a>1-n$ will be denoted as $L^p\log^a L(\Omega)$. 

For a an Orlicz function $P$ now, the Orlicz-Sobolev space $W^{1,P}(\Omega)$ is simply the set of functions that belong in $W^{1,1}_{loc}(\Omega)$ and their weak partial derivatives are also in $L^P(\Omega)$. It is easy to see that the vector valued versions of the above spaces are simply the maps whose coordinate functions are in the  corresponding  space of real valued functions.

For the proof of Theorem \ref{thm0} we require the following result, see \cite[Theorem 7.5.1]{Iwaniec2001}.

\begin{lemma}\label{lemma0}
	Let $u\in W^{1,P}(\Omega)$, where $P(t)=t^n/ \log (e+t)$. If $u$ is a weakly monotone function, then $u$ has a continuous representative.  
\end{lemma}
For the proof of Theorem \ref{thm1} we are going to need the following lemma which is an amalgamation of Theorem B and C in \cite{Kauhanen1999}.
\begin{lemma}\label{lemma1}
	Let $g:\Omega\to\mathbb{R}^m$ be a function in the Orlicz-Sobolev space $W^{1,P}(\Omega,\R^m)$, where $P(t)=t^n \log^a(e+t)$ with $a>n-1$. Then $g$ has a continuous and almost everywhere differentiable representative. Moreover, this representative has Lusin's (N) property.	
\end{lemma}
We note that the aforementioned theorems in \cite{Kauhanen1999} are stated about functions in  the Lorentz space. However, there is an equivalent way of defining Lorentz spaces using Orlicz spaces which we have used in the above lemma. We refer to the discussion preceding Theorems B and C in \cite{Kauhanen1999} for more details.

Due to the above Lemma, to prove Theorem \ref{thm1} it is enough to show that the coordinate functions $f_i$ of our finite distortion $\omega$-curve $f$ are in the right Orlicz-Sobolev space. To prove  that, we require another important result about the higher integrability of finite distortion maps proven in \cite{Faraco2005}, compare this with Theorems \ref{thm1/2} and \ref{thmhigh}.
\begin{lemma}\label{lemma2}
	Let $g\colon\Omega\to\R^n$, $n\geq 2$ be a mapping of finite distortion. Assume that the distortion function $K_g(x)$ satisfies $\exp(\beta K_g)\in L^1_{loc}$, for some $\beta>0$. Then there is a constant $c_1=c_1(n)$, $0<c_1<1$ such that
	\begin{equation}\label{eq12}
	J_g(x)\log^a(e+J_g(x))\in L^1_{loc}(\Omega) \ \text{and} \ |Dg(x)|^n\log^{a-1}(e+|Dg(x)|)\in L^1_{loc}(\Omega)
	\end{equation} 
for all $ a< c_1\beta$.			
	\end{lemma}

We are going to need the following important inequality. Its proof can be found in \cite[Appendix]{Faraco2005} or \cite[Lemma 6.2]{Hencl2014}.
\begin{lemma}\label{lemmaineq}
		Let $x,\ y\geq1$ then for any $a> -1$, $b>0$ we have
	\begin{equation}
		xy\log^{a}(C(n)(xy)^{1/n})\leq \frac{C(n)}{b}x\log^{a+1}(x^{1/n})+C(a,b,n)\exp(by),
	\end{equation}
	
\end{lemma}

We also require a slight variant of the above inequality. We include a proof here for completeness.
\begin{lemma}\label{lemma3}
	Let $x,\ y\geq1$ then for any $a>-1$, $b>0$ we have
	\begin{equation}
		xy\log^{a}(e+(xy)^{1/n})\leq \frac{C_1(n)}{b}x\log^{a+1}(e+x^{1/n})+C_2(a,b,n)\exp(by),
		\end{equation}
	where  $C_1(n)$, $C_2(a,b,n)$ constants.
\end{lemma}
\begin{proof}
 If $x< e^{by/2}$ then since there is a $C_2=C_2(a,b,n)$ such that \[y\log^{a}(e+(xy)^{1/n})\leq C_2 e^{by/2}\] we obtain that 
 \begin{equation}
 		xy\log^{a}(e+(xy)^{1/n})\leq C_2 e^{by}
 \end{equation}
If on the other hand,  $x\geq e^{by/2}$ then $y\leq \frac{2}{b}\log x$. Hence we have that
 \begin{align}
y\log^{a}(e+(xy)^{1/n}) &\leq  \frac{2}{b}\log x\log^{a}\left(e+\left(x\frac{2}{b}\log x\right)^{1/n}\right)\\ &\leq \frac{2}{b}\log (e+x)\log^{a}\left(e+\left(x^{2/b+1}\right)^{1/n}\right)\\&\leq \left(\frac{b+2}{b}\right)^{a}\frac{2n}{b}\log (e+x)^{1/n}\log^{a}\left(e+x^{1/n}  \right)\\&\leq C_1(a,b,n) \log^{a+1}(e+x^{1/n}),
\end{align}
where $C_1(a,b,n)=\left(\frac{b+2}{b}\right)^{a}\frac{2n}{b}$. Thus in that case we have 
\begin{equation}
	xy\log^{a}(e+(xy)^{1/n})\leq C_1(a,b,n) x\log^{a+1}(e+x^{1/n})
\end{equation}
 and we are done.
	\end{proof}

	\section{Higher integrability of finite distortion curves}\label{higherint}
	
	In this section we give the proofs of Theorems \ref{thm1/2} and \ref{thmhigh} which combined with Lemma \ref{lemma1} will give us Theorems \ref{thm1} and \ref{thm3/4}. 
\begin{proof}[Proof of Theorem \ref{thm1/2}]
	Let $K$ be a compact subset of $\Omega$ and $P(t)=t^n\log^a(e+t)$. We want to find the values of  $a$ for which  $\int_{K}P(\star f^*\omega)dx<\infty$.  To that end, write \begin{equation}
		\int_{K}P(\star f^*\omega)dx=I_1+I_2,
	\end{equation}
	where 
	\begin{equation}
		I_1=\int_{K\setminus (\cup_i V_i)}P(\star f^*\omega)dx \ \text{and} \ I_2=\int_{(\cup_i V_i)\cap K}P(\star f^*\omega)dx.
	\end{equation}
Here the sets $V_i$, $i=1,2,\dots$ are open subsets covering the set $E_f$ from condition (D) such that $\overline{V_i}\subset B_i$ where $B_i$ are the balls of condition (D).
	
	We now want to estimate the integrals $I_1$ and $I_2$. We start with $I_1$. Notice that if $\omega_y=\sum_I\phi_I(y)dx_I$, $y\in\R^m$ then $\star f^*\omega=\sum_I\phi(f(x)) \star f^*dx_I(x)$ with $x\in \R^n$. Moreover, since the $\omega$ form is bounded we have that $(||\omega||\circ f)(x)\geq |\omega|_{inf}$ and $\sum_I\phi(f(x))\leq |\omega|_{\ell_1}$. By condition (D) now we know that on the set $K\setminus (\cup_i V_i)$ the functions $\star f^*dx_I$, $I$ any multi-index, are uniformly bounded by a constant $M_2>0$ almost everywhere. Hence 
	\begin{equation}
		I_1\leq \int_{K\setminus (\cup_i V_i)}P\left(C \right)dx,
	\end{equation}
	where $C=C(\omega)>0$ constant. Now, the integral on the right hand side is finite since  $K$ is compact.
	
	To estimate $I_2$ now notice first that 
	\begin{equation}\label{eq22}
		I_2\leq\sum_i^\infty\int_{V_i\cap K}P(\star f^*\omega)dx.
	\end{equation}
	On each of the sets $V_i$ now there is a multi-index $I_i$ such that $\star f^*dx_{I_i}\geq \max_I\{\star f^*dx_I\}$. Since there are finitely many multi-indices, let us enumerate them as $I_j$ with $j=1,\dots, \binom{m}{n}$. We can now consider the finitely many sets $A_j=\left(\bigcup_{k}V_k\right)\cap K$, where $k$ runs over all the sets $V_i$ where $\star f^*dx_{I_j}\geq\max_I\{\star f^*dx_I\}$. We can now rewrite \eqref{eq22} as 
	\begin{equation}\label{eq24}
		I_2\leq\sum_j^{\binom{m}{n}}\int_{A_j}P(\star f^*\omega)dx.
	\end{equation}
	We fix a $j$ now and notice that $A_j\subset W_j:=(\bigcup_k B_k )\cap K_\varepsilon$, where $k$ runs over the balls $B_i$ for which $\star f^*dx_{I_j}\geq\max_I\{\star f^*dx_I\}$ and $K_\varepsilon$ is a $\varepsilon$-neighbourhood of $K$. On $W_j$ the distortion inequality \eqref{eq03} becomes 
	\begin{equation}\label{eq23}
		\left(||\omega||\circ f\right)(x)|Df(x)|\leq K_f(x) \star f^*dx_{I_j}\sum_I\phi_I(f(x)) 
	\end{equation}
	Notice now that if we set $f_{I_j}=(f_{j_1},\dots,f_{j_n})$, where $I_j=(j_1,\dots,j_n)$, then by \eqref{eq23} and using the facts that the $\omega$ form is bounded and $|Df_{I_j}(x)|\leq|Df(x)|$ we obtain 
	\begin{equation}
		|Df_{I_j}(x)|\leq \frac{|\omega|_{\ell_1}}{|\omega|_{\inf}} K_f(x) \star f^*dx_{I_j}, \ \text{for almost every}\ x\in W_j.
	\end{equation} 
	Notice that $|\omega|_{\ell_1}/|\omega|_{\inf}>1$. Hence that above equation implies that $f_{I_j}$ is a mapping of finite distortion on $W_j$ with distortion function $K_{f_{I_j}}=|\omega|_{\ell_1}/|\omega|_{\inf} K_f(x)$. By assumption now we know $\exp(\lambda K_f)\in L^1_{loc}(\Omega)$ or equivalently \[\exp(\lambda \frac{|\omega|_{inf}}{|\omega|_{\ell_1}}K_{f_{I_j}})\in L^1_{loc}(\Omega), \ \text{for some} \ \lambda>0.\]   Hence by Lemma \ref{lemma2} we obtain 
	\begin{equation}\label{eq301}
		\star f^*dx_{I_j}\log^{a}(e+\star f^*dx_{I_j})\in L^1_{loc}(W_j),
	\end{equation}
	for all $a<c_1\lambda \frac{|\omega|_{inf}}{|\omega|_{\ell_1}}$. Hence $I_2<\infty$ for those $a$, as we wanted.

	Next we want to show that $\int_{K}\hat{P}(|Df(x)|)dx<\infty$, where $\hat{P}(t)=t^n\log^{a-1}(e+t)$, for all $a<c_1\lambda \frac{|\omega|_{inf}}{|\omega|_{\ell_1}}$. By the distortion inequality we obtain that 
	\begin{equation}
		\int_{K}\hat{P}(|Df(x)|)dx\leq \int_{K}\hat{P}\left(\frac{|\omega|_{\ell_1}}{|\omega|_{\inf}}K_f(x)\star f^*\omega\right).
	\end{equation}
	By Lemma \ref{lemma3} we have that for all $b>0$ 
	\begin{equation}\label{eq25}
		\hat{P}\left(\frac{|\omega|_{\ell_1}}{|\omega|_{\inf}}K_f(x)\star f^*\omega\right)\leq \frac{C_1}{b}\star f^*\omega\log^{a}\left(e+\star f^*\omega^{1/n}\right)+C_2\exp\left(b\frac{|\omega|_{\ell_1}}{|\omega|_{\inf}}K_f(x)\right).
	\end{equation}
	By what we have proven so far, we have that the first term of the right hand side is integrable on $K$ when $a<c_1\lambda \frac{|\omega|_{inf}}{|\omega|_{\ell_1}}$.	 Moreover, for suitable $b>0$ we have that $C_2\exp\left(b\frac{|\omega|_{\ell_1}}{|\omega|_{\inf}}K_f(x)\right)\in L^1(K)$. Hence, \eqref{eq25} gives us that $\int_{K}\hat{P}(|Df(x)|)dx<\infty$, for $a<c_1\lambda \frac{|\omega|_{inf}}{|\omega|_{\ell_1}}$   as we wanted.
	\end{proof}
The classical way to prove higher integrability results about quasiregular maps in through the so called weak reverse H\"older inequalities and Gehring's lemma, see for example \cite{Bojarski1983} and  \cite{Iwaniec1998}. The same method was adapted to the setting of quasiregular curves by Pankka and Onninen in \cite{Onninen2021}. Similarly the proof of  higher integrability for finite distortion maps, in \cite{Faraco2005} was inspired by the same methods.

To prove Theorem \ref{thmhigh} we will adapt the arguments in \cite{Faraco2005} to our setting. To this end we need the  Hardy-Littlewood maximal function which is defined for any locally integrable mapping $g: \mathbb{R}^n \to \mathbb{R}$ by
\begin{equation*}
M(g)(x)= \sup_{r>0}\frac{1}{|B(x,r)|} \int_{B(x,r)} |g(y)|dy.
\end{equation*}

We start with a generalized version of Gehring's lemma which is implicitly contained in the proof of \cite[Theorem 1.1]{Faraco2005}.

\begin{lemma}[Generalized Gehring's lemma]\label{ggl}
	Let $g,K:\Omega\to\R$ be  functions such that $g\in L^1_{loc}(\Omega)$, $K(x)$ measurable and $1 \leq K(x)<\infty $ a.e. and $\exp(\beta K(x))\in L^1_{loc}$ for some $\beta>0$. Suppose that these functions satisfy the generalized weak reverse H\"older inequality
	\begin{equation}\label{eqg1}
		\fint_{\frac{1}{2}B}|g(x)|dx\leq C_1(n)\left(\fint_{B}(K(x)|g(x)|)^{s}dx\right)^{1/s},
	\end{equation}
	for some $s<1$ and all balls $B\subset\subset \Omega$. Then there is a $c_1=c_1(n)>0$ such that
	\begin{equation}\label{eqg2}
		\fint_{\frac{1}{2}B}|g(x)|\log^a\left(e+\frac{|g(x)|}{\fint_B |g(x)|dx}\right)dx\leq C_2(n,\beta, s)\fint_{B}\exp(\beta K(x))dx\fint_B |g(x)|dx,
	\end{equation}
	for all $a<c_1(n,s)\beta$.
\end{lemma}
\begin{proof}
	Fix a ball $B_0(x_0,r_0)\subset\subset \Omega$. Since both \eqref{eqg1} and \eqref{eqg2} are homogeneous with respect to $g$ we can assume without loss of generality that $\int_{B_0}|g(x)|dx=1$. Let now $d(x)=\dist(x,\R^n\setminus B_0)$ and define the functions $h_1,h_2:\R^n\to\R$ by \begin{equation}
		h_1(x)=d(x)^n |g(x)| \ \text{and} \ h_2(x)=\chi_{B_0}(x),
	\end{equation}
	for all $x\in \Omega$ and $0$ outside of $\Omega$. We are going to prove that for all balls $B\subset \R^n$ either 
	\begin{equation}\label{eqg3}
		\left(\fint_B h_1(x)dx\right)^{1/n}\leq C_1(n)\left(\fint_{2B}(K(x)h_1(x))^s\right)^{1/sn},
	\end{equation}
	in case $3B\subset B_0$ or
	\begin{equation}\label{eqg4}
		\left(\fint_B h_1(x)dx\right)^{1/n}\leq C_3(n) \left(\fint_{2B} h_2(x)dx\right)^{1/n},
	\end{equation}
	when $3B$ is not contained in $B_0$. We can assume that $B\cap B_0\not=\emptyset$ otherwise the left-hand side is zero and thus the inequalities hold trivially. So suppose that $3B\subset B_0$ then a simple geometric argument shows that \[\max_B d(x)\leq 4\min_{2B} d(x).\] Hence by \eqref{eqg1} we have,
	\begin{equation}
		\begin{aligned}
			\left(\fint_B h_1(x)dx\right)^{1/n}&\leq \max_B d(x)\left(\fint_B |g(x)|dx\right)^{1/n}\\&\leq C_1(n)\min_{2B}d(x)\left(\fint_{2B}(K(x)|g(x)|)^{s}dx\right)^{1/sn}\\ &\leq C_1(n)\left(\fint_{2B}(K(x)h_1(x))^{s}dx\right)^{1/sn}.
		\end{aligned}
	\end{equation}
	
	On the other hand, suppose that $3B$ is not contained in $B_0$. Since $B_0$ intersects $B$, again a simple geometric argument shows that \[\max_B d(x)\leq \max_{2B}d(x)\leq C_4(n) |2B\cap B_0|^{1/n}.\]
	Hence, using the fact that  $\int_{B_0} |g(x)|dx=1$ we  have that\begin{equation}
		\begin{aligned}
			\left(\frac{1}{|B|}\int_B h_1(x)dx\right)^{1/n}&\leq\max_B d(x)\left(\frac{1}{|B|}\int_{B\cap B_0} |g(x)|dx\right)^{1/n}\\&\leq C_4(n)\left(\frac{|2B\cap B_0|}{|B|}\int_{B_0} |g(x)|dx\right)^{1/n}\\&\leq C_3(n)\left(\fint_{2B} h_2(x)dx\right)^{1/n}.
		\end{aligned}
	\end{equation}
	Notice now that due to the definition of the functions $h_1$ and $h_2$ we have that the $\sup$ in the definition of the maximal functions $M(h_1)$, $M(Kh_1)$ and $M(h_2)$ is actually a $\max$. Hence, by \eqref{eqg3}, we have that when the $\sup$ is achieved for a ball $B$ with $3B\subset B_0$ then
	\begin{equation}
		M(h_1)(y)^{1/n}\leq C_1(n)M((Kh_1)^s)(y)^{1/sn}
	\end{equation}
	otherwise by \eqref{eqg4}, we have that 
	\begin{equation}
		M(h_1)(y)^{1/n}\leq C_4(n)M(h_2)^{1/n}.
	\end{equation}
	
	Notice now that $M(h_2)(y)\leq 1$ and thus there is a $\lambda_1(n)$ such that the set $\{y\in\R^n\colon C_4(n)M(h_2)(y)\geq \lambda^n\}$ is empty for every $\lambda>\lambda_1$. Hence, when $\lambda$ large we have
	\begin{equation}
		\{y\in\R^n\colon M(h_1)(y)\geq \lambda^n\}\subset \{y\in\R^n\colon C_1(n)M((Kh_1)^s)(y)\geq \lambda^{sn}\}
	\end{equation} 
	which implies that 
	\begin{equation}
		|\{y\in\R^n\colon M(h_1)(y)\geq \lambda^n\}|\leq |\{y\in\R^n\colon C_1(n)M((Kh_1)^s)(y)\geq \lambda^{sn}\}|.
	\end{equation}
	We want now to apply \cite[Proposition 2.1]{Faraco2005} twice and conclude that 
	\begin{equation}\label{eqg5}
		\int_{\{h_1>\lambda^n\}}h_1(x)dx\leq C_5(n)\lambda^{n-sn}\int_{\{C_6(n,s)Kh_1>\lambda^n\}}(K(x)h_1(x))^sdx.
	\end{equation}
	
	To do that we need to show that $h_1$ and $(Kh_1)^s$ are in $L^1(\R^n)$. This is easy to see for $h_1$ since $g\in L^1_{loc}$. For $(Kh_1)^s$ we can argue using the inequality $	ab\leq \exp(\kappa a)+\frac{2b}{\kappa}\log(e+b/\kappa)$ and the fact that $K(x)$ is exponentially integrable and $h_1$ is in $L^1$. We refer to the  proof of Theorem \ref{thm0} for a more detailed argument. 
	
	Let now $a>0$ be a constant and set 
	\begin{equation}
		\Psi(\lambda)=\frac{n-sn}{a}\log^a(\lambda)+\log^{a-1}\lambda.
	\end{equation}
	Notice that
	\begin{equation}
		\Phi(\lambda):=\frac{d}{d\lambda}\Psi(\lambda)=\frac{n-sn}{\lambda}\log^{a-1}(\lambda)+\frac{a-1}{\lambda}\log^{a-2}\lambda
	\end{equation}
	and $\Phi(\lambda)>0$ for all $\lambda>\lambda_2(n,s)=e^{1/(n-sn)}$. Moreover, 
	\begin{equation}
		\lambda^{n-sn} \Phi(\lambda)=\frac{d}{d\lambda}(\lambda^{n-sn}\log^{a-1}\lambda).	
	\end{equation}
	We multiply now both sides of \eqref{eqg5} by $\Phi(\lambda)$ and we integrate with respect to $\lambda$ over the interval $(\lambda_0,j)$, where $\lambda_0=\max\{\lambda_1,\lambda_2\}$ and $j>\lambda_0$ large. Notice that the functions $\Phi(\lambda)h_1(x)$ and $(K(x)h_1(x))^s\lambda^{s-sn}\Phi(\lambda)$ are non-negative on the sets we are integrating over. Hence, by the Fubini-Tonelli theorem we can change the order of integration to obtain 
	\begin{equation}
		\begin{aligned}
		\int_{A_1}h_1(x)\int_{\lambda_0}^{h_1^{1/n}}&\Phi(\lambda)d\lambda dx+\int_{A_1'}h_1(x)\int_{\lambda_0}^{j}\Phi(\lambda)d\lambda dx \\ &\leq C_5(n)\int_{A_2}(K(x)h_1(x))^s\int_{\lambda_0}^{(C_6Kh_1)^{1/n}}\lambda^{n-sn}\Phi(\lambda)d\lambda dx\\&+C_5(n)\int_{A_2'}(K(x)h_1(x))^s\int_{\lambda_0}^{j}\lambda^{n-sn}\Phi(\lambda)d\lambda dx,
	\end{aligned}
	\end{equation}
	where $A_1=\{\lambda_0<h_1^{1/n}<j\}$, $A_1'=\{h_1^{1/n}\geq j\}$ and $A_2=\{\lambda_0<(C_6Kh_1)^{1/n}<j\}$, $A_2'=\{(C_6Kh_1)^{1/n}\geq j\}$.

 Notice now that 
	\begin{equation}
		\begin{aligned}\int_{A_2}(K(x)h_1(x))^s\int_{\lambda_0}^{j}\lambda^{n-sn}\Phi(\lambda)d\lambda dx\leq \int_{A_2}K(x)h_1(x)\log^{a-1}(C_6K(x)h_1(x))^{1/n} dx,
		\end{aligned}
	\end{equation} and thus we obtain 
	\begin{equation}
		\begin{aligned}
		\int_{A_1}&h_1(x)(\Psi(h_1^{1/n})-\Psi(\lambda_0)) dx+\int_{A_1'}h_1(x)(\Psi(j)-\Psi(\lambda_0)) dx\\&\leq C_5(n)\left(\int_{A_2}C_6(n)K(x)h_1(x)\log^{a-1}(C_6K(x)h_1(x))^{1/n} dx +\int_{A_2'}(K(x)h_1(x))^sj^{n-sn}\log^{a-1}j dx\right),
		\end{aligned}
	\end{equation}
	which implies
	\begin{equation}\label{eqg6}
		\begin{aligned}
			\frac{n-sn}{a}&\left(\int_{A_1}h_1(x)\log^a(h_1^{1/n}) dx+\int_{A_1'}h_1(x)\log^a j dx\right)\\&\leq C_7(n)I_1 +C_5(n)I_2+C_8(n,s,a)\left(\int_{A_1}h_1(x)dx+\int_{A_1'}h_1(x)dx\right),
		\end{aligned}
	\end{equation}
where $$I_1=\int_{A_2}K(x)h_1(x)\log^{a-1}(C_6K(x)h_1(x))^{1/n}$$ and $$I_2=\int_{A_2'}(K(x)h_1(x))^sj^{n-sn}\log^{a-1}j dx.$$
	Notice now that since $\int_{B_0}|g(x)|dx=1$ and $d(x)\leq r_0$ we have that \[\int_{A_1}h_1(x)dx+\int_{A_1'}h_1(x)dx\leq 2\int_{B_0}d(x)^n |g(x)|dx\leq C_9(n)|B_0|.\] Hence, \eqref{eqg6} becomes
	\begin{equation}\label{eqg10}
		\begin{aligned}
			\frac{1}{a}&\left(\int_{A_1}h_1(x)\log^a(h_1^{1/n}) dx+\int_{A_1'}h_1(x)\log^a j dx\right)\\&\leq C_{10}(n,s)I_1+C_{11}(n,s)I_2+C_{12}(n,s,a)|B_0|.
		\end{aligned}
	\end{equation}

	Notice now that $\{h_1>\lambda_0^n\}\subset \{C_6(n)K(x)h_1(x)>\lambda_0^n\}$ since its not hard to see that $C_6(n)>1$. Hence we can write
\begin{equation}
	I_1= \int_{E_1}K(x)h_1(x)\log^{a}(C_6K(x)h_1(x))^{1/n} dx+\int_{E_2}K(x)h_1(x)\log^{a}(C_6K(x)h_1(x))^{1/n} dx,
\end{equation}
where $E_1=A_2\cap \{h_1^{1/n}\leq \lambda_0\}$ and $E_2=\{h_1^{1/n}>\lambda_0\}\cap\{(C_6Kh_1)^{1/n}<j\}\subset A_1$.  The first integral, on the right hand side, is bounded by $C_{13}(n,s)\int_{B_0}e^{\beta K(x)}$. We will call the second integral $P_0$. 

Moreover, we can do the same for the integral $I_2$ and write is as a sum of integrals over the sets $A_1'$, $A_2'\setminus A_1'\cap \{h_1^{1/n}\leq\lambda_0\}$ and $A_2'\setminus A_1'\cap \{h_1^{1/n}>\lambda_0\}\subset A_1$ which we call $P_1$, $P_2$ and $P_3$ respectively. Again the integral over the set $A_2'\setminus A_1'\cap \{h_1^{1/n}\leq\lambda_0\}$ is bounded by $C_{14}(n,s)\int_{B_0}e^{\beta K(x)}$. Furthermore, notice that \[P_3\leq \int_{A_1}C_6^{n-sn}K(x)h_1(x)\log^{a-1}(C_6K(x)h_1(x))^{1/n}=:P_4.\]
Hence \eqref{eqg10} gives us
\begin{equation}\label{eqg11}
	\begin{aligned}
			\frac{1}{a}&\left(\int_{A_1}h_1(x)\log^a(h_1^{1/n}) dx+\int_{A_1'}h_1(x)\log^a j dx\right)\\&\leq C_{10}(n,s)P_0+C_{11}(n,s)(P_1+P_4)+C_{15}(n,s)\int_{B_0}e^{\beta K(x)}+C_{12}(n,s,a)|B_0|.
	\end{aligned}
\end{equation}
	We apply now Lemma \ref{lemmaineq}  and  we obtain
	\begin{equation}
		P_0\leq C(n,a)\int_{A_1}e^{\beta K(x)}dx+\frac{C_6^{1/n}}{\beta}\int_{A_1}h_1(x)\log^{a}(h_1(x))^{1/n} dx
	\end{equation}
 and
 	\begin{equation}
 	P_4\leq C'(n,a)\int_{A_1}e^{\beta K(x)}dx+\frac{C_6^{n-sn+1/n}}{\beta}\int_{A_1}h_1(x)\log^{a}(h_1(x))^{1/n} dx
 \end{equation}
and
\begin{equation}
	P_1\leq C''(n,a)\int_{A_1'}e^{\delta K(x)^s}dx+\frac{1}{\delta}\int_{A_1'}h_1^s(x)j^{n-sn}\log^{a-1}j\log\left(h_1^{s/n}j^{1-s}\log^{(a-1)/s}j\right) dx,
\end{equation}
for some $\delta>0$ which will be fixed later. We will call the second integral on the right hand side $P_5$. Using the the above estimates and the fact that $A_1,\ A_1'\subset B_0$ and $e^{\delta K(x)^s}<ce^{\beta K(x)}$ for some constant $c=c(\delta)>0$,  \eqref{eqg11} gives 
	\begin{equation}\label{eqg7}
		\begin{aligned}
		\frac{1}{a}&\left(\int_{A_1}h_1(x)\log^a(h_1^{1/n}) dx+\int_{A_1'}h_1(x)\log^a j dx\right)\\&\leq \frac{C_{16}(n,s)}{\beta}\int_{A_1}h_1(x)\log^{a}(h_1(x))^{1/n} dx+\frac{C_{11}(n,s)}{\delta}P_5 +C_{17}(n,s,a,\beta)\int_{B_0}e^{\beta K(x)}dx\\  &+C_{12}(n,s,a)|B_0|,
		\end{aligned}
	\end{equation}
Put now $a=\frac{\beta}{2C_{16}(n,s)}$. It is not hard to see  that for all $j$ and for large enough $\delta>0$ we have that \[\frac{C_{11}}{\delta}P_5\leq \frac{1}{a}\int_{A_1'}h_1(x)\log^a j dx.\] 
Hence, \eqref{eqg7} implies 
 
\begin{equation}
	\begin{aligned}
		\int_{A_1}h_1(x)\log^a(h_1^{1/n}) dx\leq  C_{17}(n,s,a,\beta)\int_{B_0}e^{\beta K(x)}dx +C_{12}(n,s,a)|B_0|.
	\end{aligned}
\end{equation}
It is easy to see that $|B_0|<\int_{B_0}e^{\beta K(x)}dx$. Hence, if we take limits as $j\to\infty$ then by the monotone convergence theorem we obtain
 \begin{equation}\label{eqg8}
	\int_{\{h_1>\lambda_0^n\}}h_1(x)\log^a(h_1^{1/n}) dx\leq C_{18}(n,s,\beta)\int_{B_0}e^{\beta K(x)}dx.
\end{equation}

We are now ready to prove \eqref{eqg2}. First notice that since $\lambda_0^n>1$ there is a constant $k>0$ such that $x\log^a(e+x)\leq kx\log^a(x)$, for all $x>\lambda_0^n$. Now notice that
\begin{equation}
	\int_{B_0}h_1(x)\log^a(e+h_1) dx=\int_{\{h_1>\lambda_0^n\}}h_1(x)\log^a(e+h_1) dx+ \int_{\{h_1\leq\lambda_0^n\}\cap B_0}h_1(x)\log^a(e+h_1) dx.
\end{equation}
The second integral on the right-hand side is bounded by $C_{19}(n,s)|B_0|$. Thus 
\begin{equation}
	\int_{B_0}h_1(x)\log^a(e+h_1) dx\leq	k\int_{\{h_1>\lambda_0^n\}}h_1(x)\log^a(h_1) dx+ C_{19}(n,s)|B_0|.
\end{equation}
Using \eqref{eqg8} we obtain 
\begin{equation}
	\int_{B_0}h_1(x)\log^a(e+h_1) dx\leq C_{20}(n,s,\beta)\int_{B_0}e^{\beta K(x)}dx.
\end{equation}
Notice now that in the ball $\frac{1}{2}B_0$ we have that $d(x)^n\geq \frac{1}{2^n}r_0^n\geq C_{23}(n)|B_0|$. Hence the above inequality implies
\begin{equation}
	C_{21}(n)|B_0|\int_{\frac{1}{2}B_0}|g(x)|\log^a(e+C_{21}|B_0| |g(x)|) dx	\leq C_{22}(n,s,\beta)\int_{B_0}e^{\beta K(x)}dx.
\end{equation}
Using now the normalization $\int_{B_0}|g(x)|dx=1$ and the inequality $c\log^a(e+x)<\log^a(e+cx),$ $c<1$ in the case that $C_{21}(n)<1$ we obtain 
\begin{equation}
	\fint_{\frac{1}{2}B_0}|g(x)|\log^a\left(e+\frac{|g(x)|}{\fint_{B_0} |g(x)|dx}\right)dx\leq C_2(n,\beta, s)\fint_{B_0}\exp(\beta K(x))dx\fint_{B_0} |g(x)|dx,
\end{equation}
as we wanted.
\end{proof}
Next we have the analogues of the  weak reverse H\"older inequalities for finite distortion curves.

\begin{lemma}[Generalized weak reverse H\"older inequality]\label{grhi}
	Let $f\in W^{1,n}_{loc}(\Omega,\R^m)$ be a finite distortion $\omega$-curve, where $\omega$ is an $n$-volume form with constant coefficients and let $B\subset\subset \Omega$. Then 
	\begin{equation}
		\fint_{\frac{1}{2}B}\star f^*\omega dx\leq C(n,m)\left(\fint_{B}(K_f(x)\star f^*\omega)^{n/(n+1)}dx\right)^{(n+1)/n}.
	\end{equation}
\end{lemma}
\begin{proof}
	By following the argument in the proof of  \cite[Lemma 6.1]{Onninen2021}, where the authors prove a Cacciopoli type inequality for quasiregular $\omega$-curves, we can show that for all balls $B(y,r)\subset\subset\Omega$ and all functions $\phi\in C^\infty_{0}(B)$ we have
	\begin{equation}
		\int_B \phi(x) (\star f^*\omega)(x)dx\leq ||\omega||\int_B |\nabla\phi(x)||f(x)-f_B| |Df(x)|^{n-1}dx,
	\end{equation}
	where $f_B=\fint_B f(x)dx$ (integration is meant coordinate-wise). It is important to note that this argument uses the fact that $f\in W^{1,n}_{loc}(\Omega,\R^m)$. We choose now $\phi$ so that $\phi(x)=1$ when $x\in \frac{1}{2}B$, $0\leq \phi\leq1$ and $|\nabla\phi|\leq 2/r$. With this choice for $\phi$ and by applying H\"older's inequality we obtain 
	\begin{equation}\label{eqr1}
		\int_{\frac{1}{2}B}  (\star f^*\omega)(x)dx\leq \frac{2||\omega||}{r}\left(\int_B|f(x)-f_B|^{n^2}dx\right)^{1/n^2}\left(\int_B|Df(x)|^{n^2/(n+1)}\right)^{n^2-1/n^2}.
	\end{equation} 
	By applying now the Poincare-Sobolev inequality, see for example \cite[Theorem A.18]{Hencl2014},  coordinate-wise it is easy to see that \begin{equation}
		\left(\int_B|f(x)-f_B|^{n^2}dx\right)^{1/n^2}\leq C_0(n,m)\left(\int_B|Df(x)|^{n^2/(n+1)}\right)^{(n+1)/n^2}.
	\end{equation}
	Hence, \eqref{eqr1} becomes
	\begin{equation}
		\fint_{\frac{1}{2}B}  (\star f^*\omega)(x)dx\leq ||\omega||C(n,m) \left(\fint_B|Df(x)|^{n^2/(n+1)}\right)^{(n+1)/n}.
	\end{equation}
	Applying the distortion inequality \eqref{eq03} now in the above inequality gives
	\begin{equation}
		\fint_{\frac{1}{2}B}  (\star f^*\omega)(x)dx\leq C(n,m)\left(\fint_B(K_f(x)\star f^*\omega)^{n/(n+1)}\right)^{(n+1)/n},
	\end{equation}
	which is what we wanted.
\end{proof}

The proof of Theorem \ref{thmhigh} now readily follows.
\begin{proof}[Proof of Theorem \ref{thmhigh}]
	The fact that \begin{equation}\label{eqproof}
		\star f^*\omega\log^a(e+\star f^*\omega)\in L^1_{loc}(\Omega)
	\end{equation} follows immediately from Lemmas \ref{ggl} and \ref{grhi}. On the other hand, \[|Df(x)|^n\log^{a-1}(e+|Df(x)|)\in L^1_{loc}(\Omega)\] follows
	by using \eqref{eqproof} and the inequality in Lemma \ref{lemmaineq}. 
\end{proof}

\section{Proofs of Theorems \ref{thm1}, \ref{thm3/4} and \ref{thmmodulus}}\label{proofs}
	\begin{proof}[Proof of Theorem \ref{thm1}]
		Let $U_k=\{x\in\Omega\colon \dist(x,\partial \Omega)>1/k \}\cap B(0,k)$, where $k=1,2,3,\dots$ and $B(0,k)$ open ball centred at $0$ and of radius $k$. If we take $c=c(n,\omega)=\frac{n|\omega|_{\ell_1}}{c_1|\omega|_{inf}}$ then we can find a $p$ so that $n-1<p$ and $p+1<\frac{\lambda|\omega|_{inf}c_1}{|\omega|_{\ell_1}}$ since $\lambda>c$.  By Theorem \ref{thm1/2} now, we have that  $f$ is in $W^{1,P}(U_k)$ for $P(t)=t^n\log^p(e+t)$ and some $p>n-1$. Hence, Lemma \ref{lemma1}  implies that  $f$ has a continuous and almost everywhere differentiable representative in $U_k$, which also has Lusin's (N) property, for each large enough $k$. Since $U_k\subset U_{k+1}$ and $\bigcup_{k=1}^\infty U_k=\Omega$ this immediately implies that $f$ has a  representative in $\Omega$ with all the above properties. 
	\end{proof}
\begin{proof}[Proof of Theorem \ref{thm3/4}]
	The proof follows in the same way as that of Theorem \ref{thm1} only instead of Theorem \ref{thm1/2} we use \ref{thmhigh}.
	\end{proof}
	
\begin{proof}[Proof of Theorem \ref{thmmodulus}]
	First we assume that $F$ is a ball compactly contained in $\Omega$. We want to apply \cite[Theorem 3.6]{Donaldson1971} to the coordinate functions of $f$ in $F$. To do that we need to know that  they are in $W^{1,P}_{loc}(\Omega)$, with $P(t)$ such that \[\int_{1}^\infty \frac{P^{-1}(t)}{t^{(n+1)/n}}dt<\infty.\]  It is not hard to see now that this last condition is satisfied when $P(t)=t^n\log^a(e+t)$, with $a>n$. By Theorems \ref{thm1/2} and \ref{thmhigh} now there is a constant $q$ such that if $\exp(\lambda K_f(x))\in L^1_{loc}(\Omega)$ for some $\lambda>q$ then we have that $f\in W^{1,P}_{loc}(\Omega)$. Thus by applying \cite[Theorem 3.6]{Donaldson1971} to the coordinate functions it is easy to see that  \eqref{eqmodulus} holds and we are done. For a general compact set $F$ now the result follows routinely. We sketch the proof for completeness. Cover the segment connecting $x$ and $y$  by a finite number of balls $B_i\subset \Omega$, $i=1,\dots, k$ (if the segment is not in $\Omega$ just consider a polygonal path) and find points $x_i$, $i=1,\dots, k$ such that $x_0=x$, $x_k=y$ and $x_i\in B_i\cap B_{i+1}$, $i=1,\dots,k-1$. By applying now the special case of the theorem for balls in  $B_i$ we can conclude that \begin{equation}
		|f(x_i)-f(x_{i+1})|\leq Q_i ||f||_{W^{1,P}}\int_{|x_i-x_{i+1}|^{-n}}^\infty \frac{P^{-1}(t)}{t^{(n+1)/n}},
	\end{equation} 
for $i=1,\dots, k$.	Triangle inequality and the fact that $|x_i-x_{i+1}|^{-n}>|x-y|^{-n}$ now, give us the required inequality. 
	\end{proof}

		\section{Monotonicity implies continuity}\label{monotonicity}
		Here we prove Theorem \ref{thm0}. The proof is essentially an adaptation of  \cite[Lemma 2.8]{Hencl2014}.
\begin{proof}[Proof of Theorem \ref{thm0}]
	From Lemma \ref{lemma0} it is enough to show that the coordinate functions are in $W^{1,P}(\Omega)$, where $P(t)=t^n/\log(e+t)$. For that it is enough to show that $|Df|\in L^n \log^{-1}L(\Omega)$.
	
	First notice that because the function $P(t),\ t>0$ is increasing and because $\omega$ is a bounded form we have that 
	\begin{align}
		\frac{|Df|^n}{\log(e+|Df|)}&\leq \frac{c_0K_f(x)\star f^*\omega}{\log(e+(c_0K_f(x)\star f^*\omega)^{1/n})}\\ &\leq \frac{c_0K_f(x)\star f^*\omega}{\log(e+(c_0\star f^*\omega)^{1/n})}\\ &\leq n\frac{c_0K_f(x)\star f^*\omega}{\log(e+c_0\star f^*\omega)},
	\end{align}
	where $c_0>0$ is a constant that depends on $\omega$. Hence for any compact set $K\subset \Omega$ we have that 
	\begin{equation}\label{eq001}
		\int_K \frac{|Df|^n}{\log(e+|Df|)}dx\leq n\int_K \frac{c_0K_f(x)\star f^*\omega}{\log(e+c_0\star f^*\omega)}dx
	\end{equation}
	By using now \eqref{eq001} and the inequality (see \cite[Lemma 2.7]{Hencl2014} for a proof)
	\begin{equation}
		ab\leq \exp(\kappa a)+\frac{2b}{\kappa}\log(e+b/\kappa), \ \text{for all} \ a\geq 1, b\geq 0, \kappa>0
	\end{equation}
	for $a=K_f(x)$, $b=\frac{c_0\star f^*\omega}{\log(e+c_0\star f^*\omega)}$ and $\kappa=\lambda$ we obtain
	\begin{align}
		\int_K \frac{|Df|^n}{\log(e+|Df|)}dx\leq n&\int_K \exp(\lambda K_f(x))dx\\ + &\frac{2n}{\lambda}\int_K  \frac{c_0\star f^*\omega}{\log(e+c_0\star f^*\omega)}\log\left( e+\frac{c_0\star f^*\omega}{\lambda\log(e+c_0\star f^*\omega)}\right)dx
	\end{align} 
	The first integral is now finite by assumption. We can split the second integral into two integrals, over the sets \[A_1=\{x\in K\colon \lambda\log(e+c_0\star f^*\omega)\leq 1\}\]
	and $A_2=K\setminus A_1$. On $A_2$ the integrand is dominated by $c_0\star f^*\omega$ which is integrable, while on $A_1$ notice that $\star f^*\omega\leq e^{1/\lambda-1}/c_0$ and since the function $t/\log(e+t)$ is increasing we obtain that the integrand is bounded in $A_1$ and we are done.
\end{proof}
	\section{Counterexamples to continuity and Lusin's condition}\label{counterexamples}
	
	\begin{proof}[Proof of Theorem \ref{thm2}]
	 Let $z=(x,y)$ and $f(x,y)=(f_1(x,y),f_2(x,y),f_3(x,y))$, where we take \begin{equation}
			f_1(z)=
			\begin{cases}
				\frac{x}{|z|}\left[e+\log\left(\frac{1}{|z|}\right)\right]^{-1} & \text{for} |z|\leq 1\\
				x/e &\text{for} |z|>1
			\end{cases}
		\end{equation}
		
		\begin{equation}
			f_2(z)=\begin{cases}
				\frac{y}{|z|}\left[e+\log\left(\frac{1}{|z|}\right)\right]^{-1} & \text{for} |z|\leq 1\\
				y/e &\text{for} |z|>1
			\end{cases}
		\end{equation}	 
		and
		\begin{equation}
			f_3(z)=\log\log \left(e+\left|\log|z|\right|\right).
		\end{equation}
		Let $F(z)=(f_1(z),f_2(z))$. This map is then a mapping of finite distortion, meaning that $F\in W^{1,1}_{loc}$ and that $|DF|^2\leq K'(z) J_F$, for some function  $K'(z)$ which is finite almost everywhere. Moreover it is true that $e^{\lambda K'(z)}\in L^1_{loc}(\mathbb{R}^2)$, for all $\lambda<2$. In fact we have that 
		\begin{equation}
			|DF(z)|=\begin{cases}
				\frac{1}{|z|\left(e-\log|z|\right)} & \text{for} |z|\leq 1\\
				1/e & \text{for} |z|>1
			\end{cases}	,
		\end{equation} 
		
		\begin{equation}
			J_F=\begin{cases}
				\frac{1}{|z|^2\left(e-\log|z|\right)^3} &\text{for} |z|\leq 1\\
				1/e^3 & \text{for} |z|>1
			\end{cases}
		\end{equation}
		and 
		\begin{equation}
			K'(z)=\begin{cases}
				\left(e-\log|z|\right) &\text{for} |z|\leq 1\\
				e & \text{for} |z|>1
			\end{cases}
		\end{equation}
		Moreover, 
		\begin{equation}
			|Df_3(z)|\leq 	\frac{1}{|z|\left(e-\log|z|\right)\log(e-\log|z|)}, \hspace{2mm} \text{for} |z|\leq 1
		\end{equation}
		and 
		\begin{equation}
			|Df_3(z)|\leq\frac{1}{|z|\left(e+\log|z|\right)\log(e+\log|z|)}, \hspace{2mm} \text{for} |z|>1
		\end{equation}

		It is also easy to see that $\star f^{*}\omega= J_F$, $||\omega||=1$ and that $(f,\omega)$ satisfies condition (D). Putting everything together  we obtain
		\begin{align*}
			|Df|^2\leq \left(|DF|+|Df_3|\right)^2&\leq \left(\sqrt{K'(z)J_F}+\frac{\sqrt{K'(z)J_F}}{\log(e+|\log|z||)}\right)^2\\&=\left(\frac{1}{\log(e+|\log|z||)}+1\right)^2K'(z)J_F.
		\end{align*}
		So that \eqref{eq01} is true with $K(z)=\left(\frac{1}{\log(e+|\log|z||)}+1\right)^2K'(z)$. Notice now that for any $\varepsilon>0$ there is an $r(\varepsilon)=r>0$ such that $\frac{1}{\log(e+|\log|z||)}<\varepsilon$, for $|z|\leq r$. Also, notice that $K(z)$ is bounded in $\R^2\setminus B(0,r)$.  Hence, \begin{equation}
			\int_{\R^2}\exp(\lambda K(z))dz=\int_{\R^2\setminus B(0,r)}\exp(\lambda K(z))dz+\int_{B(0,r)}\exp(\lambda K(z))dz
		\end{equation}
	and the first integral is finite for all $\lambda>0$ while $$\int_{B(0,r)}\exp(\lambda K(z))dz\leq \int_{B(0,r)}\exp(\lambda(1+\varepsilon)^2 K'(z))dz$$ which is finite for all $\lambda<2/(1+\varepsilon)^2$. Since $\varepsilon>0$ can be arbitrary we have 
 $e^{\lambda K(z)}\in L^1_{loc}$, for any given $\lambda< 2$, however the map $f$ does not have a continuous representative since its third coordinate $f_3$ does not ($f_3$ has a logarithmic singularity at $0$). 
	\end{proof}

For the construction in Theorem \ref{thmlusin} we are going to need a special case of \cite[Lemma 5.1]{Kauhanen1999}. 

\begin{lemma}\label{ulemma}
 There exists a function $u\in W^{1,n}_{loc}(B(0,1))$ so that $u$ is non negative, radial, continuous outside the origin, tends to $\infty$ when $x\to 0$ and satisfies 
	\begin{equation}\label{eqlem00}
	\int_{B(0,1)}\Phi(|\nabla u(x)|)dx<\infty 
	\end{equation}
where $\Phi(t)=t^n\log^{n-1}(e+t)\log^{n-1}(\log(e+t))$ and
\begin{equation}\label{eqlem01}
	 |\nabla u(x)|\leq \frac{1}{|x|(1-\log|x|)\log|\log |x||}	, \ \text{when}\ |x| \ \text{is small enough}.
	\end{equation}
\end{lemma} 
\begin{proof}
	Let $\phi(t)=t^{-n}\log^{-n}(e+t)\log^{-n}(\log(e+t))$. We define the real functions \begin{equation}
		h_k(t):=\inf\{s>0:\phi(2s)\leq (2^kt)^n\}, \ t\in(0,\infty).
	\end{equation} 
Notice now that 
\begin{equation}\label{eqlem02}
	h_k(t)\leq \frac{1}{t(1-\log t)\log|\log t|},
\end{equation} for all $k\geq 1$ and all $t$ small enough. Indeed, it is enough to show that \begin{equation}
	\phi\left(\frac{2}{t(1-\log t)\log|\log t|}\right)\leq (2^kt)^n,
\end{equation}
which is equivalent to 
\begin{equation}
	\frac{(1-\log t)\log|\log t|}{\log\left(e+\frac{2}{t(1-\log t)\log|\log t|}\right)\log\left(\log\left(e+\frac{2}{t(1-\log t)\log|\log t|}\right)\right)}\leq 2^{k-1}.
\end{equation}
Notice now that the left hand side goes to $1$ as $t\to 0$ so the above inequality is true for all $k\geq 1$ and $t$ small enough. Next we wish to show that $\int_0^1h_k(t)dt=\infty$. To that end, notice that we can find a $\sigma_k$ such that $\phi(2\sigma_k)<2^{kn}$ which implies that 
\begin{equation}
	\{(t,s)\colon s>\sigma_k,\ 0<t<2^{-k}\phi^{1/n}(2s)\}\subset \{(t,s)\colon 0<t<1,\ 0<s\leq h_k(t)\}.
\end{equation}
Hence by Fubini's theorem we have that 
\begin{equation}
	\begin{aligned}
	\int_0^1h_k(t)dt&=\int_{\{(t,s)\colon 0<t<1,\ 0<s\leq h_k(t)\}}dtds\\&\geq \int_{	\{(t,s)\colon s>\sigma_k,\ 0<t<2^{-k}\phi^{1/n}(2s)\}}dtds\\&=2^{-k}\int_{\sigma_k}^\infty \phi^{1/n}(2s)ds.
\end{aligned}
\end{equation}
It is not hard to see now that the last integral is $\infty$. Hence we can define a decreasing sequence $\{a_k\}$ of positive real number such that $a_1=1$ and 
\begin{equation}
	\int_{a_{k+1}}^{a_k}h_k(x)dx=1, \ \text{for all} \ k\in\N.
\end{equation}
The above implies that $h_k(a_{k+1})\geq 1$ and thus $\phi(2)\geq 2^{nk}a^n_{k+1}$. Hence $a_k\to 0$. We set now 
\begin{equation}
	u(x)=k+\int_{|x|}^{a_k}h_k, \ \text{when} \ a_{k+1}\leq |x|\leq a_k,
\end{equation}
which is continuous outside the origin, radial, non negative and  $\lim_{x\to0}u(x)=\infty$. Moreover, \eqref{eqlem02} implies 
\begin{equation}
	|\nabla u(x)|= |h_k(|x|)|\leq\frac{1}{|x|(1-\log|x|)\log|\log |x||},
\end{equation}
which proves \eqref{eqlem01}.
Finally, since $\phi(h_k(t))\leq (2^kt)^n$ we have that 
\begin{equation}
	h_k(t)^n\log^{n-1}(e+h_k(t))\log^{n-1}(\log(e+h_k(t)))\leq (2^kt)^{1-n}h_k(t) 
\end{equation}
and thus 
\begin{equation}
	\begin{aligned}
		\int_{B(0,1)}\Phi(|\nabla u(x)|)dx\leq & \sum_k \int_{\{a_{k+1}\leq |x|\leq a_k\}}(2^kt)^{1-n}h_k(|x|)dx\\=&\sum_k 2\pi\int_{a_{k+1}}^{a_k}r^{n-1}(2^kr)^{1-n}h_k(r)dx\\=& 2\pi\sum_k 2^{k(1-n)}<\infty. 
	\end{aligned}
\end{equation}
Hence \eqref{eqlem00} is proven and since $t^n\leq \Phi(t)$, for all $t>0$ we have that $|\nabla u(x)|^n\leq\Phi(|\nabla \phi(x)|)$ which implies that $u\in W^{1,n}_{loc}(B(0,1))$.
	\end{proof}

\begin{proof}[Proof of Theorem \ref{thmlusin}]
 We define $f=(f_1,f_2,f_3,f_4)$ where $F(z)=(f_1(z),f_2(z))$, $z=(x,y)$ will be a mapping of finite distortion.
The function $G=(f_3,f_4)$ will be constructed by adapting  the construction in \cite[Theorem 5.2]{Kauhanen1999}. 

We start by constructing $F$ first.

\textit{Construction of $F$:} The function $F$ will be obtained as the limit of a sequence of finite distortion mappings. First we set $z=(x,y)$ and \[g(z)=\begin{cases}
	\frac{z}{|z|}\left[1+\log\left(\frac{1}{|z|}\right)\right]^{-1} & \text{for} |z|\leq 1\\
	z &\text{for} |z|>1
\end{cases}.\]
Similarly as in the proof of Theorem \ref{thm2} it is easy to calculate that \begin{equation}
	|Dg(z)|=\begin{cases}
		\frac{1}{|z|\left(1-\log|z|\right)} & \text{for} |z|\leq 1\\
		1 & \text{for} |z|>1
	\end{cases}	,
\end{equation} 

\begin{equation}
	J_g=\begin{cases}
		\frac{1}{|z|^2\left(1-\log|z|\right)^3} &\text{for} |z|\leq 1\\
		1 & \text{for} |z|>1
	\end{cases}
\end{equation}
and 
\begin{equation}
	K_g(z)=\begin{cases}
		\left(1-\log|z|\right) &\text{for} |z|\leq 1\\
		1 & \text{for} |z|>1
	\end{cases}.
\end{equation}
Take now 4 disjoint balls $B_{i,0}(c_{i,0},R_0)$, $i=1,2,3,4$, with centres $c_{i,0}$ in the segment $[-1,1]\times\{0\}$ and radii to be determined later. Let also $A_{i,0}(c_{i,0},R_0,r_0)$ denote the annuli of the same centres, outer radius $R_0$ and inner radius $r_0$ which is also to be defined later. Define now the finite distortion map
\[f_1(z)=\begin{cases}
	R_0(1-\log(R_0))g(z-c_{i,0}) & \text{for} |z|\in A_{i,0}(c_{i,0},R_0,r_0)\\
\frac{R_0(1-\log R_0)}{r_0(1-\log r_0)}z &\text{for} |z|\in B(c_{i,0},r_0)\\
z & \text{otherwise}
\end{cases}\]

Inductively now, inside each of the balls $B_{i,n-1}(c_{i,n-1},r_{n-1})$, $i=1,\dots, 4^{n}$ take 4 disjoint balls with centres on the segment $[-1,1]\times\{0\}$. Thus we obtain the collection of balls $B_{i,n}(c_{i,n},R_n)$, $i=1,\dots 4^{n+1}$. We define the finite distortion map 
\[f_{n+1}(z)=\begin{cases}
	R_n(1-\log(R_n))g(z-c_{i,n}) & \text{for} |z|\in A_{i,n}(c_{i,n},R_n,r_n)\\
	\frac{R_n(1-\log R_n)}{r_n(1-\log r_n)}z &\text{for} |z|\in B(c_{i,n},r_n)\\
	z & \text{otherwise}
\end{cases}.\]

Let $F_n(z)=f_1\circ\dots\circ f_n$ which will be a sequence of finite distortion mappings. It is easy to see now that $F_n$ converges in the $W^{1,2}$-norm to a finite distortion function $F\in W^{1,2}_{loc}(\R^2,\R^2)$. We want $F$ to have exponentially integrable distortion. It is easy to see that when $\lambda<2$ we have
\begin{equation}
	\begin{aligned}
	\int_{\R^2}\exp(\lambda K_F(z))dz&=\sum_{n=0}^{\infty}4^{n+1}\int_{A_{i,n}(c_{i,n},R_n,r_n)}\exp(\lambda K_{f_{n+1}}(z))dz\\&=\sum_{n=0}^{\infty}4^{n+1}\int_{0}^{2\pi}\int_{r_n}^{R_n}e^{\lambda}e^{-\lambda\log r}rdrd\theta\\&=\sum_{n=0}^{\infty}\frac{2\pi 4^{n+1}e^\lambda}{-\lambda+2}(R_n^{-\lambda+2}-r_n^{-\lambda+2})\\&\leq \sum_{n=0}^{\infty}\frac{2\pi 4^{n+1}e^\lambda}{-\lambda+2}R_n^{-\lambda+2}.
		\end{aligned}
\end{equation} 
Thus if we take $R_n\leq \left(\frac{-\lambda+2}{2\pi e^\lambda 8^n}\right)^{1/(-\lambda+2)}$ then $\sum_{n=0}^{\infty}\frac{2\pi 4^{n+1}e^\lambda}{-\lambda+2}R_n^{-\lambda+2}<\infty$ and $F$ will be exponentially integrable for any $\lambda<2$. Finally, it is easy to calculate that for $z\in A_{i,n}(c_{i,n},R_n,r_n)$ we have that 
\begin{equation}
	J_F(z)=a_n\frac{1}{|z-c_{i,n}|^2(1-\log|z-c_{i,n}|)^3},
\end{equation}
and
\begin{equation}
	K_F(z)=1-\log|z-c_{i,n}|,
\end{equation}
where \[a_n=\prod_{j=0}^{n-1}\frac{R_j^2(1-\log R_j)^2}{r_j^2(1-\log r_j)^2}R_n^2(1-\log R_n)^2.\]

Notice that so far we have no restriction placed on $r_n$ and $R_n$ other than $R_n\leq \left(\frac{-\lambda+2}{2\pi e^\lambda 8^n}\right)^{1/(-\lambda+2)}$ and $r_n< R_n$.

The function $G$ that we are going to construct now will have the property that $G\in W^{1,2}_{loc}(\R^2,\R^2)$ and for any fixed $\varepsilon>0$ and suitable sequences $r_n$ and $R'_n<R_n$ it will satisfy
\begin{equation}
	|DG(z)|\leq \sqrt{a_n}\frac{\varepsilon}{|z-c_{i,n}|(1-\log|z-c_{i,n}|)}= \varepsilon \sqrt{K_F(z)J_F(z)},
\end{equation}
for all $z\in A_{i,n}(c_{i,n},R'_n,r_n)$ and $|DG(z)|=0$ otherwise. Moreover, we will have that $G([-1,1]\times\{0\})=[-1,1]^2$. Assuming that we have constructed this function, we will then have that the map $f=(F,G)\in W^{1,2}_{loc}(\R^2,\R^2)$ will satisfy the distortion inequality
\begin{equation}
	|Df(z)|\leq \left(|DF(z)|+|DG(z)|\right)^2\leq (1+\varepsilon)^2K_F(z)J_F(z)=K_f(z)\star f^*\omega,
	\end{equation}
almost everywhere. Since $K_F$ is exponentially integrable, so is $K_f$. Since $\varepsilon$ can be arbitrarily small $K_f$ can be arbitrarily close to $\lambda$-exponentially integrable. Moreover, it is easy to see that $F([-1,1]\times\{0\})=[-1,1]\times\{0\}$. Hence, $f([-1,1]\times\{0\})$ will be a surface with non zero Hausdorff 2 measure which means $f$ does not have Lusin's (N) property. All that remains is to construct G.

\textit{Construction of $G$:} We will construct recursively a sequence of functions $G_n$ in $W^{1,2}_{loc}(\R^2,\R^2)$ which will converge uniformly to the desired function $G$. We start by partitioning the square $[-1,1]^2$ in dyadic squares. We set \[S_0=\{[-1,0]\times[0,1],[-1,0]\times[-1,0],[0,1]\times[0,1],[0,1]\times[-1,0]\}\] and more generally $S_n$ will be a collection of $4^{n+1}$ squares obtained by partitioning each of the squares of $S_{n-1}$ into $4$. By $K_{i,n}$, $i=1,\dots 4^{n+1}$ we will denote the elements of $S_n$ and by $w_{i,n}$ their centres. Also we let $w_{0,-1}=0$.

We proceed now with the definition of $\{G_n\}$. First we define $G_0=0$. Assuming that we have defined $G_0,\dots,G_n$ we are going to construct $G_{n+1}$. We are going to use the sequences $R_n$ and $r_n$ from the construction of $F$. First we can choose $R_0$ small enough so that the two dimensional version of function $u$ from Lemma \ref{ulemma} satisfies \eqref{eqlem01} when $|z|\leq R_0$. Moreover, since $\int_{B(0,1)}\Phi(|\nabla u(z)|)dz<\infty$ we can choose $R_n$ small enough so that 
\begin{equation}\label{eqlus2}
	\int_{B(0,R_n)}\Phi(|\nabla u(z)|)dz<\frac{1}{2^{n+1}4^{n+1}},
	\end{equation}
for all $n=0,1,\dots$. Finally, for any fixed $\varepsilon>0$ we first temporarily choose $r_n=R'_n$ and then take $R'_n<R_n$ small enough so that \begin{equation}\label{eqlus1}
	\frac{1}{\log|\log|z-c_{i,n}||}\leq \varepsilon\sqrt{a_n},
\end{equation} when $|z-c_{i,n}|\leq R'_n$. This is possible since $a_n$ is constant on the annulus $A_{i,n}(c_{i,n},R_n,r_n)$.  Next since $u(z)\to\infty$ as $z\to 0$ we can in fact reselect $r_n<R'_n$  so that 
\begin{equation}
	u(r_n,r_n)-u(R'_n,R'_n)=d_n,
\end{equation}
where $d_n$ denotes the distance between the centres of a square in $S_n$ and one of its subdividing squares in $S_{n+1}$, $K_{j,n+1}\subset K_{i,n}$. Remember there is no restriction on $r_n$ and decreasing it makes $a_n$ larger so \eqref{eqlus1} is not affected by that. Define now $G_{n+1}$ to be $G_n$ outside of the balls $B_{i,n}(c_{i,n},R'_n)$. When $r_n\leq|z-c_{i,n}|\leq R'_n$ define \begin{equation}
	G_{n+1}(z)=w_{j,n-1}+(u(z-c_{i,n})-u(R'_n,R'_n))\frac{w_{i,n}-w_{j,n-1}}{|w_{i,n}-w_{j,n-1}|},
\end{equation}
where $j$ is the unique index such that $K_{i,n}\subset K_{j,n-1}$. Finally, when $|z-c_{i,n}|\leq r_n$ we define \[G_{n+1}(z)=w_{i,n}.\]
Next we need to show that the sequence $G_n$ converges uniformly. To prove that, notice that $G_n$ and $G_{n+1}$ differ only inside the balls $B_{i,n}(c_{i,n},R'_n)$ where $G_n$ is constant while $G_{n+1}$ maps the ball to a segment of length $d_n$. Hence, for $m>n$
\begin{equation}
	||G_{m}-G_n||_{\infty}\leq \sum_{k=n}^{m-1}||G_{k+1}-G_k||_\infty \leq \sum_{k=n}^{m-1}d_k.
\end{equation}
Since the lengths $d_k$ shrink at a geometric rate we have that $G_n$ is a uniformly Cauchy sequence and we are done. Thus the limit map $G$ will be continuous and it is easy to see that \eqref{eqlus2} implies \begin{equation}
	\sum_{i=1}^{4^{n+1}}\int_{B_{i,n}(c_{i,n},R'_n)}\Phi(|DG_{n+1}(z)|)dz<\frac{1}{2^{n+1}},
\end{equation} which in turn implies that $G\in W^{1,2}(\R^2,\R^2)$. Moreover, from \eqref{eqlem01} and \eqref{eqlus1} we obtain that 
\begin{equation}
	|DG(z)|=|\nabla u(z)|\leq \frac{\varepsilon \sqrt{a_n}}{|z-c_{i,n}|(1-\log|z-c_{i,n}|)},
\end{equation}
when $r_n\leq|z-c_{i,n}|\leq R'_n$. Finally, it is easy to see that for any $w_{i,n}$ there exists a sequence of points $z_k\in A_{i,k}(c_{i,k},R'_k,r_k)$, converging to some point in the segment $[-1,1]\times\{0\}$, such that $G(z_k)=w_{i,n}$. Since the set of all centres $\{w_{i,n}\}_{i,n\in \N}$ is dense in the square $[-1,1]^2$ and since $G$ is continuous we have that \[G([-1,1]\times\{0\})=[-1,1]^2,\]
as we wanted.
	\end{proof}
\bibliographystyle{amsplain}
\bibliography{bibliography}

\providecommand{\bysame}{\leavevmode\hbox to3em{\hrulefill}\thinspace}
\providecommand{\MR}{\relax\ifhmode\unskip\space\fi MR }
\providecommand{\MRhref}[2]{%
  \href{http://www.ams.org/mathscinet-getitem?mr=#1}{#2}
}
\providecommand{\href}[2]{#2}
\begin{thebibliography}{10}

\bibitem{Astala2009}
Kari Astala, Tadeusz Iwaniec, and Gaven Martin, \emph{Elliptic partial
  differential equations and quasiconformal mappings in the plane}, Princeton
  University Press, 2009.

\bibitem{Bojarski1983}
B.~Bojarski and T.~Iwaniec, \emph{Analytical foundations of the theory of
  quasiconformal mappings in $\mathbb{R}^n$}, Annales Academiae Scientiarum
  Fennicae Series A I Mathematica \textbf{8} (1983), 257--324.

\bibitem{Donaldson1971}
Thomas~K Donaldson and Neil~S Trudinger, \emph{Orlicz-sobolev spaces and
  imbedding theorems}, Journal of Functional Analysis \textbf{8} (1971), no.~1,
  52--75.

\bibitem{Faraco2005}
Daniel Faraco, Pekka Koskela, and Xiao Zhong, \emph{Mappings of finite
  distortion: the degree of regularity}, Advances in Mathematics \textbf{190}
  (2005), no.~2, 300--318.

\bibitem{Heikkilae2021a}
Susanna Heikkilä, \emph{Signed quasiregular curves}, preprint,
  arXiv:2101.09943, 2021.

\bibitem{Heikkilae2021}
Susanna Heikkilä, Pekka Pankka, and Eden Prywes, \emph{Quasiregular curves of
  small distortion in product manifolds}, preprint, arXiv:2106.11871, 2021.

\bibitem{Hencl2014}
Stanislav Hencl and Pekka Koskela, \emph{Lectures on mappings of finite
  distortion}, Springer International Publishing, 2014.

\bibitem{Iwaniec2001}
T.~Iwaniec and G.~Martin, \emph{Geometric function theory and non-linear
  analysis}, Oxford mathematical monographs, Oxford University Press, 2001.

\bibitem{Iwaniec1998}
Tadeusz Iwaniec, \emph{The gehring lemma}, Quasiconformal Mappings and
  Analysis, Springer New York, 1998, pp.~181--204.

\bibitem{Iwaniec2002}
Tadeusz Iwaniec, Pekka Koskela, and Gaven Martin, \emph{Mappings of
  {BMO}-distortion and beltrami-type operators}, Journal
  d{\textquotesingle}Analyse Math{\'{e}}matique \textbf{88} (2002), no.~1,
  337--381.

\bibitem{Iwaniec2001a}
Tadeusz Iwaniec, Pekka Koskela, and Jani Onninen, \emph{Mappings of finite
  distortion: Monotonicity and continuity}, Inventiones Mathematicae
  \textbf{144} (2001), no.~3, 507--531.

\bibitem{Kauhanen1999}
Janne Kauhanen, Pekka Koskela, and Jan Mal{\'{y}}, \emph{On functions with
  derivatives in a lorentz space}, manuscripta mathematica \textbf{100} (1999),
  no.~1, 87--101.

\bibitem{Kauhanen2001}
\bysame, \emph{Mappings of finite distortion: condition {N}}, Michigan
  Mathematical Journal \textbf{49} (2001), no.~1.

\bibitem{Kauhanen2003}
Janne Kauhanen, Pekka Koskela, Jan Mal{\'{y}}, Jani Onninen, and Xiao Zhong,
  \emph{Mappings of finite distortion: Sharp orlicz-conditions}, Revista
  Matem{\'{a}}tica Iberoamericana (2003), 857--872.

\bibitem{StevenG.Krantz2008}
Steven~G. Krantz and Harold~R. Parks, \emph{Geometric integration theory},
  Birkhäuser Boston, MA, August 2008.

\bibitem{Onninen2021}
Jani Onninen and Pekka Pankka, \emph{Quasiregular curves: Hölder continuity
  and higher integrability}, Complex Analysis and its Synergies \textbf{7}
  (2021), no.~3-4.

\bibitem{Pankka2020}
Pekka Pankka, \emph{Quasiregular curves}, Ann. Acad. Sci. Fenn. Math.
  \textbf{45} (2020), no.~2, 975--990.

\bibitem{Reshetnyak1989}
Yu. Reshetnyak, \emph{Space mappings with bounded distortion}, American
  Mathematical Society, 1989.

\bibitem{Rickman}
S.~Rickman, \emph{Quasiregular mappings}, vol.~26, Ergeb. Math. Grenzgeb.,
  no.~3, Springer-Verlag, Berlin, 1993.

\bibitem{RobertA.Adams2003}
John J. F.~Fournier Robert A.~Adams, \emph{Sobolev spaces}, Academic
  Press/Elsevier, 2003.

\bibitem{Vodopyanov1977}
S.~K. Vodop'yanov and V.~M. Gol'dshtein, \emph{Quasiconformal mappings and
  spaces of functions with generalized first derivatives}, Siberian
  Mathematical Journal \textbf{17} (1977), no.~3, 399--411.

\end{thebibliography}

\end{document}